\documentclass[reqno]{amsart}

\usepackage{amsmath,amssymb,amsthm}

\usepackage{tikz}
\usepackage{caption}
\usepackage{graphicx}
\usepackage{graphics}

\newtheorem*{thm}{Theorem}
\newtheorem*{lemma}{Lemma}

\newcommand{\diam}{\operatorname{diam}}

\begin{document}

\title[]{Quadratic Crofton and sets that see \\themselves as little as possible}

\author[]{Stefan Steinerberger}
\address{Department of Mathematics, University of Washington, Seattle, WA 98195, USA} \email{steinerb@uw.edu}

\thanks{S.S. is supported by the NSF (DMS-2123224) and the Alfred P. Sloan Foundation.}

\begin{abstract} Let $\Omega \subset \mathbb{R}^2$ and let $\mathcal{L} \subset \Omega$ be a one-dimensional set with finite length $L =|\mathcal{L}|$. We are interested in minimizers of an energy functional that measures the size of a set projected onto itself in all directions: we are thus asking for
sets that see themselves as little as possible (suitably interpreted). Obvious minimizers of the functional are subsets
of a straight line but this is only possible for $L \leq \mbox{diam}(\Omega)$. The problem has an equivalent formulation: the expected number of intersections between a random line and $\mathcal{L}$ depends only on the length of $\mathcal{L}$ (Crofton's formula). We are interested in sets $\mathcal{L}$ that minimize the variance of the expected number of intersections. We solve the problem for convex $\Omega$ and slightly less than half of all values of $L$: there, a minimizing set is the union of copies of the boundary and a line segment.
\end{abstract}
\maketitle

\section{Introduction and Results}
\subsection{Introduction} The purpose of this short paper is to introduce a problem at the interface of the calculus of variations and integral/random geometry.
The problem is easily stated in any dimension (see below) but since we have already found the case of convex domains $\Omega \subset \mathbb{R}^2$ to be challenging, we will mostly focus on that case. Let us first consider a one-dimensional rectifiable set $\mathcal{L} \subset \mathbb{R}^2$ with positive length $L$. We define a notion of energy as
$$ E(\mathcal{L}) = \int_{\mathcal{L}} \int_{\mathcal{L}}\frac{\left|\left\langle n(x), y - x \right\rangle  \left\langle  y - x, n(y) \right\rangle   \right| }{\|x - y\|^{3}}~d \sigma(x) d\sigma(y),$$
where $x,y$ are elements in the set, $n(x)$ and $n(y)$ denote the normal vectors in $x$ and $y$, respectively, and $\sigma$ is the arclength measure. This functional first arose in work of Chang, Dabrowski, Orponen and Villa \cite[Appendix A.1]{chang} on sets with nearly maximal Favard length and later in a different context in work of the author \cite{stein}.
\begin{center}
\begin{figure}[h!]
\begin{tikzpicture}[scale=3]
\draw [thick] (0,0) to[out=30, in=150] (3,0);
\filldraw (0.6, 0.28) circle (0.025cm);
\filldraw (2.4, 0.28) circle (0.025cm);
\draw[thick, ->] (0.6, 0.28) -- (0.52, 0.5);
\node at (0.7, 0.5) {$n(x)$};
\draw[thick, ->] (2.4, 0.28) -- (2.48, 0.5);
\node at (2.65, 0.5) {$n(y)$};
\node at (0.65, 0.2) {$x$};
\node at (2.3, 0.2) {$y$};
\draw [dashed] (0.6, 0.28) -- (2.4, 0.28);
\end{tikzpicture}
\end{figure}
\end{center}
\vspace{-10pt}
One way of thinking about the functional is that it measures the behavior of the set when projected onto itself in the following sense:  consider $x,y \in \mathcal{L}$ and let us take small neighborhoods around $x$ and $y$ (we may think of these as approximately being short line segments). 
We could then ask for the expected size of the projection of one such line segment onto the other under a `random' projection. Equivalently, we can ask for the likelihood that a `random' line intersects both line segments.
 For the canonical choice of `random' line (the one that is invariant under rotation and translation, also known as the kinematic measure), the arising geometric expression is exactly the integrand in the energy functional.
The quantity measures whether it is easy to see one neighborhood from the other. In order for them to be `nearly' invisible, it suffices if one (or both) of the line segments have a normal vector that is nearly orthogonal to the direction of line of sight $x-y$. Moreover, the likelihood decreases with distance.
The functional vanishes if $\mathcal{L}$ is a subset of a line. It therefore makes sense to restrict $\mathcal{L}$ to lie in
some fixed bounded domain $\Omega \subset \mathbb{R}^2$. We can now state our main problem.
\begin{quote}
\textbf{Problem.} Among all one-dimensional rectifiable $\mathcal{L} \subset \Omega \subset \mathbb{R}^2$ of fixed length, which one minimizes $E(\mathcal{L})$ and how does the minimizer depend on the enclosing set $\Omega$?
\end{quote}
An interesting aspect of the problem is that for larger $L$, the set cannot actually avoid having large projections onto itself: the set is guaranteed to see itself and the problem becomes to arrange things so that these projections are roughly of comparable size in all projections. In particular, we will prove that the functional has to grow quadratically in the length and determine the leading order and the next order term in the expansion (see \S 1.2).
For convex $\Omega \subset \mathbb{R}^2$ and suitable length $L$, the solution is relatively simple.

\begin{thm} Let $\Omega \subset \mathbb{R}^2$ be a bounded, convex domain with $C^1$-boundary. If 
$$  0 \leq L - n |\partial \Omega| \leq \diam(\Omega) \qquad \mbox{for some} \quad n \in \mathbb{N},$$
 then among all one-dimensional rectifiable sets $\mathcal{L} \subset \Omega$ of length $L$ the energy
$$ E(\mathcal{L}) = \int_{\mathcal{L}} \int_{\mathcal{L}}\frac{\left|\left\langle n(x), y - x \right\rangle  \left\langle  y - x, n(y) \right\rangle   \right| }{\|x - y\|^{3}}~d \sigma(x) d\sigma(y)$$
is minimized by the union of $n$ copies of $\partial \Omega$ and a line segment of length $L - n |\partial \Omega|$.
\end{thm}

For the unit disk $\Omega = D \subset \mathbb{R}^2$, this solves the problem for all $L$ of the form $L \in [0,2] \cup [2\pi, 2\pi + 2] \cup \dots$.
and the solution is relatively simple, however, other cases remain open. We give bounds for $E(\mathcal{L})$ in terms
of $L$ (see \S 1.2 below). Our understanding of extremal sets for other values of $L$ remains limited even for relatively simple $\Omega$ such
as the unit disk or the unit square.

\begin{center}
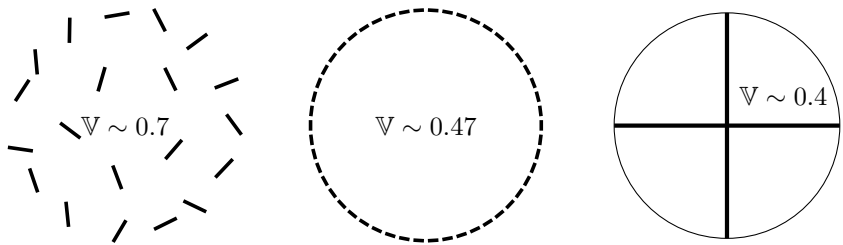
\begin{figure}[h!]
\begin{tikzpicture}
\draw (0,0) circle (1.5cm);
\draw[ultra thick] (-1.5,0) -- (1.5,0);
\draw[ultra thick] (0,-1.5) -- (0,1.5);
\node at (-8,0) {\includegraphics[width=3.2cm]{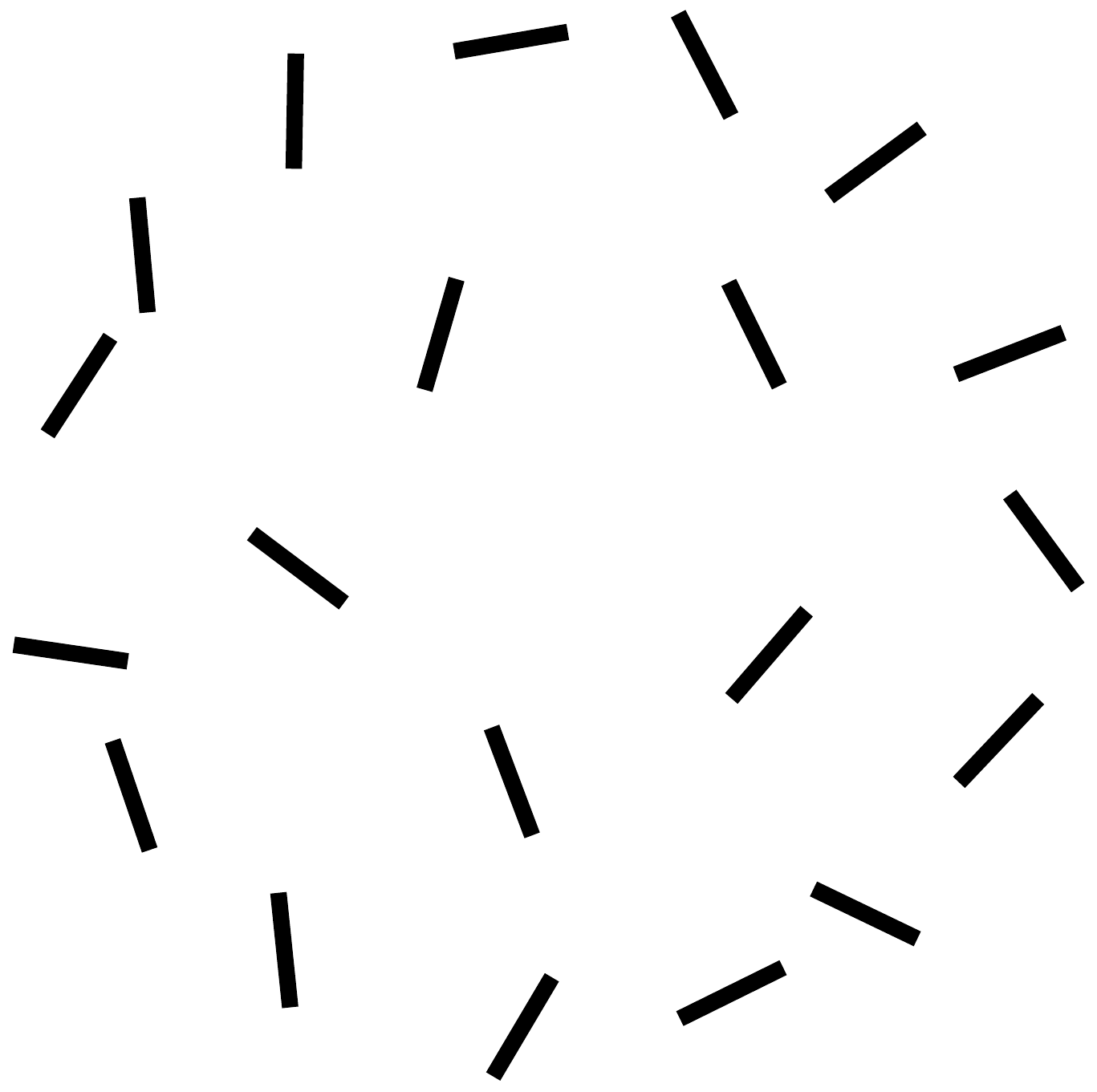}};
\node at (-8, 0) {$\mathbb{V} \sim 0.7$};
\node at (-4,0) {\includegraphics[width=3.2cm]{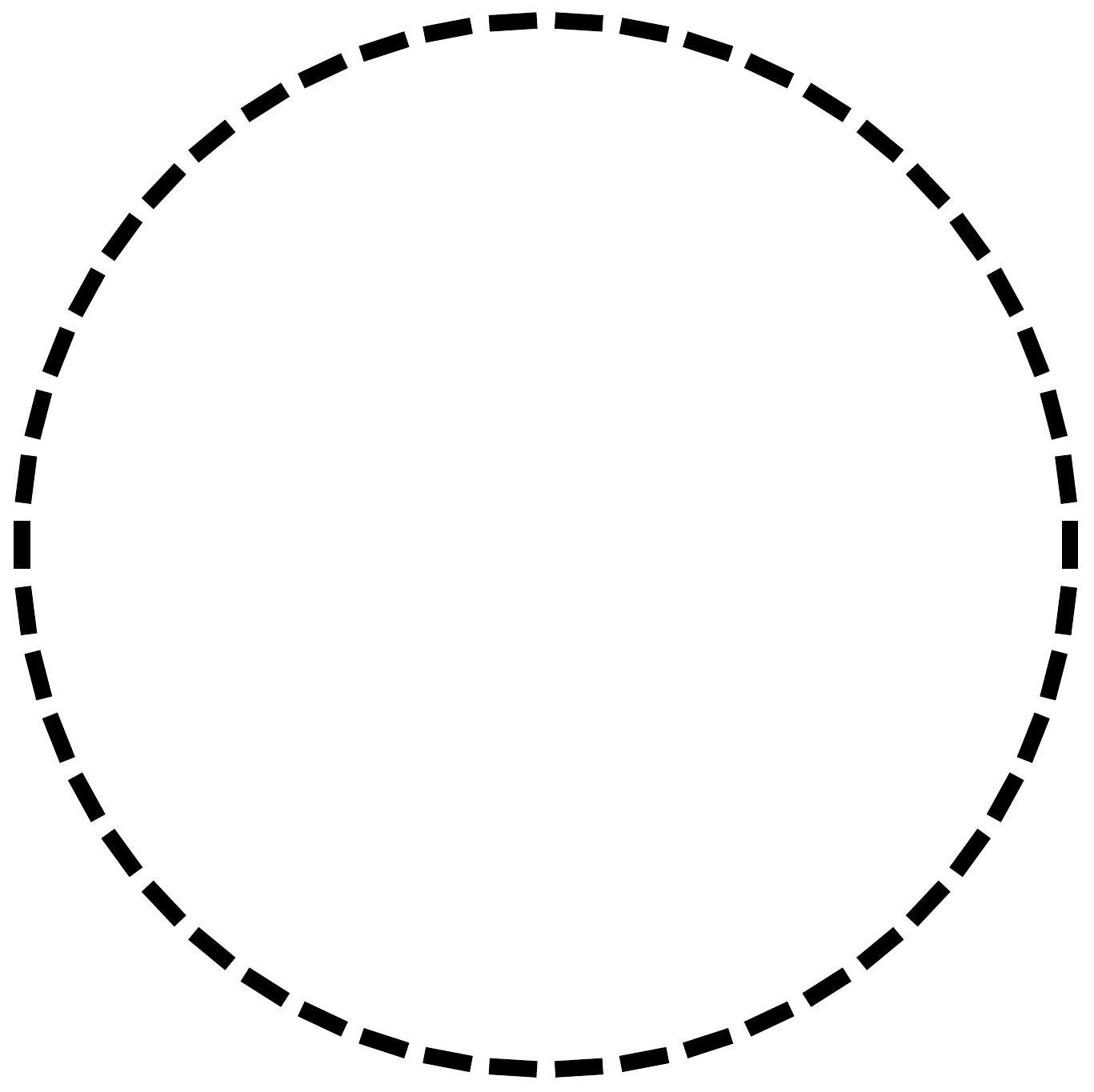}};
\node at (-4, 0) {$\mathbb{V} \sim 0.47$};
\node at (0.75,0.4) {$\mathbb{V} \sim 0.4$};
\end{tikzpicture}
\caption{Three sets of length $L=4$ in the unit disk in decreasing energy (proportional to the variance of a random variable, see \S 1.2).}
\end{figure}
\end{center}

Many other questions remain: are the solutions `periodic' in the sense that the optimal solution for $L = |\partial \Omega| + X$ is simply the optimal solution for $L = X$ and an additional copy of the boundary $\partial \Omega$?  In light of the Theorem, this is certainly conceivable (there is another viewpoint, discussed below, that could also be interpreted as supporting evidence). This would reduce the problem to the range $\diam(\Omega) \leq L \leq |\partial \Omega|$.  Another natural question is that the energy functional rewards nearby points $x,y \in \mathcal{L}$ if at least one of the tangent vectors $n(x), n(y)$ is roughly orthogonal to $x-y$. This appears to be indicative of some form of implicit regularization: one could ask, for example, whether minimizing sets $\mathcal{L}$ are necessarily the union of finitely many piecewise-differentiable curves.
Needless to say, the question is also of obvious interest when $\Omega \subset \mathbb{R}^2$ is not convex. Moreover, the question is also of obvious interest in higher dimensions and our proof generalizes.
\begin{thm} Let $\Omega \subset \mathbb{R}^n$ be a bounded, convex domain with $C^1$-boundary. There exists a constant $c_{\Omega}$ such that
if 
$$  0 \leq L - m |\partial \Omega| \leq c_{\Omega} \qquad \mbox{for some} \quad m \in \mathbb{N},$$
then among all $(n-1)-$dimensional piecewise differentiable $\Sigma \subset \Omega$ with surface area $\mathcal{H}^{n-1}(\Sigma) = L$ the
energy
$$ E(\Sigma) = \int_{\Sigma} \int_{\Sigma}\frac{\left|\left\langle n(x), y - x \right\rangle  \left\langle  y - x, n(y) \right\rangle   \right| }{\|x - y\|^{n+1}}~d \sigma(x) d\sigma(y)$$
is minimized by $m$ copies of the boundary and a segment of a hyperplane. 
\end{thm}
One admissible choice for the constant $c_{\Omega}$ is the largest $(n-1)-$dimensional volume of an intersection of $\Omega$ with a hyperplane. Just as in the two-dimensional case, our understanding of the problem outside the range covered by the Theorem is limited.

\subsection{Quadratic Crofton.} We will now give an alternative derivation that casts the question in a different light (and can help motivate some of the results). Fix a set $\mathcal{L} \subset \mathbb{R}^2$ and consider `random' lines with respect to the kinematic measure $\mu$. The kinematic measure corresponds to picking a uniform distribution over all angles and then a uniform distribution in how far the line is displaced from the origin (corresponding to the differential form $d \phi \wedge dx$). For any given line $\ell$, we denote the number of intersections of the set $\mathcal{L}$ with the line $\ell$ by $n_{\ell}(\mathcal{L})$. Crofton's formula says that the expected number of intersections only depends on the length 
$$ |\mathcal{L}| = \frac{1}{4} \int n_{\ell}(\mathcal{L}) d\mu(\ell).$$
We can use this formula to deduce that among those lines $\ell$ that intersect $\Omega$, the expected number of intersections with $\mathcal{L}$ is $2 L/ |\partial \Omega|$. This follows from applying Crofton's formula a second time, using that ($\mu-$almost all) lines intersect a convex domain exactly twice and thus
$$ \frac{1}{4} \int 1_{\ell \cap \partial \Omega \neq \emptyset} ~d\mu(\ell) =  \frac{1}{4}  \int \frac{n_{\ell}(\partial \Omega) }{2} ~d\mu(\ell) = \frac{|\partial \Omega|}{2}.$$
It does not matter how the one-dimensional set $\mathcal{L}$ is arranged, the only relevant quantity is its length.
This natural invariance is already of interest in itself and is the beginning of integral geometry (see \cite{santa, santa2}). 
 \begin{quote}
 \textbf{Problem.} Among all one-dimensional sets $\mathcal{L} \subset \Omega \subset \mathbb{R}^2$ with fixed length $L = |\mathcal{L}|$, which one minimizes the \mbox{variance} of the average number of intersections among the lines hitting $\Omega$?
 \end{quote}
 We note that the kinematic measure $\mu$ is not a probability measure: the question therefore has to be understood as the variance restricted to a compact space which we take to be the probability space of all lines that actually intersect $\Omega$ (which, through normalization, can then be turned into a probability space, see \S 2.2).
  Examples (see Fig. 1) show that this number does indeed depend strongly on $\mathcal{L}$. We start by simplifying the question a little. The variance, conditioning on those lines that intersect the domain, is proportional to
\begin{align*}
  \int 1_{\ell \cap \partial \Omega \neq \emptyset} \left( n_{\ell}(\mathcal{L}) - \frac{2L}{|\partial \Omega|}\right)^2 d\mu(\ell) &=  \int  n_{\ell}(\mathcal{L})^2 d\mu(\ell) -  \frac{4L}{|\partial \Omega|} \int n_{\ell}(\mathcal{L})d\mu(\ell) \\
 &+  \int 1_{\ell \cap \partial \Omega \neq \emptyset} \cdot \frac{4L^2}{|\partial \Omega|^2} d\mu(\ell) 
\end{align*}
which, via two applications of Crofton's formula, can be simplified to
$$  \int 1_{\ell \cap \partial \Omega \neq \emptyset} \left( n_{\ell}(\mathcal{L}) - \frac{2L}{|\partial \Omega|}\right)^2 d\mu(\ell) =   \int  n_{\ell}(\mathcal{L})^2 d\mu(\ell) - \frac{8 L^2}{|\partial \Omega|}.$$
We are thus invited to consider the problem of minimizing the energy functional $E_{\Omega}: \mathbb{R}_{\geq 0} \rightarrow \mathbb{R}_{\geq 0}$ introduced via
$$ E_{\Omega}(L) = \inf_{ |\mathcal{L}| =L} ~ \frac{1}{4} \int n_{\ell}(\mathcal{L})^2 d\mu(\ell),$$
where the infimum is taken over all 1-rectifiable sets of length $L$.
The factor $1/4$ is introduced 
for convenience (to simplify the comparison with Crofton's formula).
 A computation (carried out in \cite{stein}) shows that, for sufficiently regular $\mathcal{L} \subset \mathbb{R}^2$
 $$ \frac{1}{4} \int n_{\ell}(\mathcal{L})^2 d\mu(\ell) - |\mathcal{L}| = \frac{1}{2} \int_{\mathcal{L}} \int_{\mathcal{L}}\frac{\left|\left\langle n(x), y - x \right\rangle  \left\langle  y - x, n(y) \right\rangle   \right| }{\|x - y\|^{3}}~d \sigma(x) d\sigma(y)$$
which establishes the equivalence to the problem stated above. We can use this to deduce some immediate bounds on $E_{\Omega}(L)$.
Since
$n_{\ell}$ is integer-valued, we have
$$ E_{\Omega}(L) = \inf_{ |\mathcal{L}| =L} ~ \frac{1}{4} \int n_{\ell}(\mathcal{L})^2 d\mu(\ell) \geq \inf_{ |\mathcal{L}| =L} ~ \frac{1}{4} \int n_{\ell}(\mathcal{L}) d\mu(\ell) = L.$$
Moreover, the inequality is sharp if $\mathcal{L}$ has the property that each line intersects it at most once (for almost all lines with respect to the kinematic measure): this is the case when $\mathcal{L}$ is itself a subset of a line.  Moreover, a variance is always nonnegative and from this we deduce a second lower bound
$$ E_{\Omega}(L) = \inf_{ |\mathcal{L}| =L} ~ \frac{1}{4} \int n_{\ell}(\mathcal{L})^2 d\mu(\ell) \geq \frac{2L^2}{|\partial \Omega|}.$$
We show that this is actually fairly acurate and $\left|E_{\Omega}(L) - 2L^2/|\partial \Omega| \right| \lesssim 1$.
\begin{thm} Let $\Omega \subset \mathbb{R}^2$ be a bounded, convex domain with $C^1$ boundary. Then 
$$   \frac{|\partial \Omega|}{2} \left\{  \frac{2L}{|\partial \Omega|} \right\} \left(1 - \left\{  \frac{2L}{|\partial \Omega|} \right\}\right) \leq E_{\Omega}(L) - \frac{2L^2}{|\partial \Omega|}  \leq \frac{|\partial \Omega|}{4},$$
where $\left\{ x \right\} = x - \left\lfloor x \right\rfloor$.
The lower bound is sharp if $n |\partial \Omega| \leq L \leq n |\partial \Omega|  + \emph{diam}(\Omega)$ for some $n \in \mathbb{N}$: an extremal set is given by $n$ copies of $\partial \Omega$ and a line segment. 
\end{thm}
A natural question is whether the function $E_{\Omega}(L) - 2L^2 /|\partial \Omega|$ is $|\partial \Omega|-$periodic in $L$. A stochastic way of interpreting this is that adding another copy of the boundary does not increase the variance of the expected number of intersection (since $\mu$-a.e. line intersects $\partial \Omega$ exactly twice, there is only a shift in the expectation).

\subsection{Opaque sets and related results} For a given domain $\Omega \subset \mathbb{R}^2$, an \textit{opaque set} is a one-dimensional set $\mathcal{L}$ such that every line intersecting $\Omega$ is also intersecting $\mathcal{L}$. The obvious question is: how short can
an opaque set be? The notion itself goes back to a 1916 paper of Mazurkiewicz \cite{maz}, the term `opaque' has been introduced in a 1959 paper of Bagemihl \cite{bag}. 
The problem has proven to be notoriously difficult, even the case of $[0,1]^2$ remains unsolved and poorly understood in the sense that the lower bounds are far from the conjectured extremizer (see Fig. 2). 
The only general lower bound (see e.g. \cite{dumi, jones}) is that an opaque set in a convex domain $\Omega \subset \mathbb{R}^2$ has to have length at least  $|\partial \Omega|/2$. It appears to be tremendously difficult to improve on this lower bound: in the case of the unit square $[0,1]^2$, the lower bound of $|\partial \Omega|/2$ was first proven by Jones \cite{jones} in 1964. The currently best result is due to  Kawamura, Moriyama, Otachi \& Pach \cite{kawa} and shows that any opaque set has to have length $2.0002$. We also refer to \cite{asi, brakke, dumi2, faber,faber2}.
The connection to our problem is seen by quickly recalling the proof of the lower bound $|\partial \Omega|/2$. Using $\mathcal{O}$ to denote the opaque set, we note that each line intersecting $\Omega$ has to intersect $\mathcal{O}$ at least once and thus, with Crofton's formula, just as above
$$ |\mathcal{O}| = \frac{1}{4} \int n_{\ell}(\mathcal{O}) d\mu(\ell) \geq  \frac{1}{4} \int 1_{\ell \cap \Omega \neq \emptyset}~ d\mu(\ell) = \frac{|\partial \Omega|}{2}.$$
The only way this argument could possibly be sharp is if most lines intersecting $\mathcal{O}$ intersect it exactly once. This, in turn, would imply that the variance is also minimized. We note that this can indeed `almost' occur: if we take a highly eccentric rectangle $[0,1] \times [0, \varepsilon]$, then there is an opaque set of size $1 + 2\varepsilon$ which is quite comparable to half the boundary (having length $1 + \varepsilon$) and the opaque set has very small variance since most lines that intersect the rectangle intersect it exactly once.

\begin{center}
\begin{figure}[h!]
\begin{tikzpicture}[scale=3]
\draw [thick] (-1.5, 0.5) circle (0.5cm);
   \draw [ultra thick,domain=180:360] plot ({-1.5 + 0.5*cos(\x)}, {0.5 + 0.5*sin(\x)});
   \draw [ultra thick] (-1, 0.5) -- (-1,1);
      \draw [ultra thick] (-2, 0.5) -- (-2,1);
\draw [thick] (0,0) -- (1,0) -- (1,1) -- (0,1) -- (0,0);
\draw [ultra thick] (0, 1) -- (0.2, 0.2) -- (1, 0);
\draw [ultra thick] (0, 0) -- (0.2, 0.2);
\draw [ultra thick] (0.5, 0.5) -- (1,1);
\end{tikzpicture}
\caption{Left: an opaque set for the unit disk of length $\pi + 2 < 2\pi$. Right: the conjectured shortest opaque set for the unit square with length $\sqrt{2} + \sqrt{3/2} \sim 2.63$.}
\end{figure}
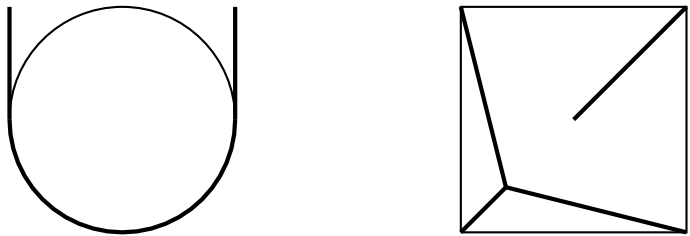
\end{center}

Unrelatedly, a recent result of the author \cite{stein} introduced this notion of energy when proving that
there exists a universal constant $c_n > 0$ depending only on the dimension so that for any bounded $\Omega \subset \mathbb{R}^n$ with $C^1-$boundary
$$ \int_{\partial \Omega \times \partial \Omega}  \frac{\left|\left\langle n(x), y - x \right\rangle  \left\langle  y - x, n(y) \right\rangle   \right| }{\|x -y\|^{n+1}}~d \sigma(x) d\sigma(y) \geq c_n |\partial \Omega|$$
with equality if and only if the domain $\Omega$ is convex. This inequality is how we originally got interested in the behavior of this energy in the first place.

\section{Proof of the Theorem}
We will prove the second formulation of the Theorem (which is the more detailed one) using the language of integral geometry. Statements for the first formulation follow easily from change of variables.

\subsection{A Lemma.} We start with a very simple inequality for random variables.
\begin{lemma}
Let $X \geq 0$ be an integer-valued random variable. Then
$$ \mathbb{E} X^2 \geq (\mathbb{E}X)^2 + \left\{ \mathbb{E} X \right\} - \left\{ \mathbb{E} X \right\}^2,$$
where $\left\{ x \right\} = x - \left\lfloor x \right\rfloor$. We have equality iff $X$ is supported on
$\left\{ \left\lfloor \mathbb{E}X \right\rfloor, \left\lfloor \mathbb{E}X \right\rfloor + 1\right\}$.
\end{lemma}
\begin{proof} Suppose $k \leq \mathbb{E} X < k+1$.
 We write
$$ \mathbb{E} X = k + \alpha.$$
There is a natural candidate for the most concentrated random variable: the one concentrated at $k$ and $k+1$. 
For that distribution, we see that $ \mathbb{P}(X = k) = 1-\alpha$ and  $\mathbb{P}(X=k+1) = \alpha.$
Then the variance is given by
$$ \mathbb{V}X = \mathbb{E}(X-\mathbb{E}X)^2 = \alpha^2(1-\alpha) +  (1-\alpha)^2 \alpha = \alpha - \alpha^2$$
and the result follows from $ \mathbb{V}X = \mathbb{E}X^2 - (\mathbb{E}X)^2.$  It remains to show that this is the
extremal case: if $\mathbb{P}(X=k) < 1-\alpha$, then we necessarily also have $\mathbb{P}(X=k+1) < \alpha$ which
shows that there has to be probability mass supported on $\left\{X < k\right\}$ and $\left\{X > k+1\right\}$. We can
move some mass towards $k$ and $k+1$ in such a way as to preserve the expectation and obtain the result from
monotonicity.
\end{proof}

\subsection{Proof of the Theorem}
\begin{proof} We separate the proof into several relatively independent parts.\\
 \textit{1. Lower Bound.} We start with Crofton's formula saying that for rectifiable $H \subset \mathbb{R}^2$
$$ |H| = \frac{1}{4} \int n_{\ell}(H) d\mu(\ell).$$
We can use this to compute the likelihood of a line hitting the convex domain $\Omega \subset \mathbb{R}^2$: since almost all
(with respect to the kinematic measure) lines that hit a convex domain hit it exactly twice, we have
$$  \mu(\Omega) =  \int 1_{\ell \cap \Omega \neq \emptyset}~ d\mu(\ell)  = \frac{1}{2}   \int n_{\ell}(\partial\Omega) ~d\mu(\ell) = 2 \cdot | \partial \Omega|.$$
This allows us to define a natural probability space: the set of all lines that intersect $\Omega$ equipped with a suitable rescaling of the kinematic measure which,
by the previous computation, is going to be 
$$ \nu(\ell) =1_{\ell \cap \Omega \neq \emptyset} \cdot \frac{\mu(\ell)}{2 |\partial \Omega|}.$$
which turns it into a probability space. Note that this implies, in particular, that
$$  \int n_{\ell}( \partial \Omega) ~d\nu(\ell) = 2.$$
Moreover, the Crofton formula applied to a subset $\mathcal{L} \subset \Omega$ transforms to
$$ L = \frac{1}{4} \int n_{\ell}(\mathcal{L}) d\mu(\ell) =\frac{|\partial \Omega|}{2} \int n_{\ell}(\mathcal{L}) \frac{d\mu(\ell)}{2 |\partial \Omega|} = \frac{|\partial \Omega|}{2} \int n_{\ell}(\mathcal{L}) d\nu(\ell)$$
which can be rewritten as
$$ \frac{2 L}{|\partial \Omega|} =   \int n_{\ell}(\mathcal{L}) ~d\nu(\ell)$$
which is the expected value of the number of intersections in this probability space.
Using the Lemma, we deduce that
\begin{align*}
  \int n_{\ell}(\mathcal{L})^2 ~d\nu(\ell) &\geq \left( \int n_{\ell}(\mathcal{L}) ~d\nu(\ell) \right)^2 + \left\{  \int n_{\ell}(\mathcal{L}) ~d\nu(\ell) \right\} - \left\{ \int n_{\ell}(\mathcal{L}) ~d\nu(\ell)\right\}^2 \\
  &\geq \frac{4L^2}{|\partial \Omega|^2} + \left\{  \frac{2L}{|\partial \Omega|} \right\} -  \left\{  \frac{2L}{|\partial \Omega|} \right\}^2.
  \end{align*}
Rescaling back, we find that
$$ \frac{1}{4} \int n_{\ell}(\mathcal{L})^2 d\mu(\ell) \geq  \frac{2L^2}{|\partial \Omega|} + \frac{|\partial \Omega|}{2}\left\{  \frac{2L}{|\partial \Omega|} \right\} -  \frac{|\partial \Omega|}{2}\left\{  \frac{2L}{|\partial \Omega|} \right\}^2$$
which is the desired lower bound. \\

\textit{2. Sharpness of the lower bound.} Let us say that $\mathcal{L}$ is comprised of $k$ copies of $\partial \Omega$ and a line segment $K$. Then $L = |\mathcal{L}| = k |\partial \Omega| + |K|$ and
\begin{align*}
\frac{1}{4} \int n_{\ell}(\mathcal{L})^2 ~d\mu(\ell) &= \frac{1}{4} \int 1_{\ell \cap \Omega \neq \emptyset} \cdot (2k + n_{\ell}(K))^2 ~d\mu(\ell) \\
&=  \frac{1}{4} \int 1_{\ell \cap \Omega \neq \emptyset} \cdot (4 k^2 +4k  n_{\ell}(K) + n_{\ell}(K)^2) ~d\mu(\ell).
\end{align*}
Since $K$ is a line segment, we have that $n_{\ell}(K)^2 = n_{\ell}(K)$ and may use this to deduce
\begin{align*}
\frac{1}{4} \int n_{\ell}(\mathcal{L})^2 ~d\mu(\ell) &=  \frac{1}{4} \int 1_{\ell \cap \Omega \neq \emptyset} \cdot (4 k^2 +4k  n_{\ell}(K) + n_{\ell}(K)) d\mu(\ell)  \\
&= k^2 \mu(\Omega) + \frac{1}{4} \int 1_{\ell \cap \Omega \neq \emptyset}  (4 k +1)   n_{\ell}(K)  d\mu(\ell).
\end{align*}
We already computed above that $\mu(\Omega) = 2 \cdot |\partial \Omega|$. Moreover,
\begin{align*}
  \frac{1}{4} \int 1_{\ell \cap \Omega \neq \emptyset}  (4 k +1)   n_{\ell}(K)  d\mu(\ell) &=  \frac{1}{4} \int   (4 k +1)   n_{\ell}(K)  d\mu(\ell) \\
  &= (4k+1) |K|.
  \end{align*}
  Altogether, we deduce, also recalling  $L = k |\partial \Omega| + |K|$ that
\begin{align*}
\frac{1}{4} \int n_{\ell}(\mathcal{L})^2 ~d\mu(\ell)  &=2k^2 |\partial \Omega| + (4k+1) |K| \\
&= \frac{2 L^2}{|\partial \Omega|} +|K| \left(1 - \frac{2 |K|}{|\partial \Omega|}\right).
\end{align*}
Since $|K| \leq \diam(\Omega) \leq |\partial \Omega|/2$, we have
$$ \left\{  \frac{2L}{|\partial \Omega|} \right\} = \left\{  \frac{2k|\partial \Omega| + 2|K|}{|\partial \Omega|} \right\} = \frac{2 |K|}{|\partial \Omega|}$$
and the lower bound proved above can be rewritten as
\begin{align*}
\frac{1}{4} \int n_{\ell}(\mathcal{L})^2 d\mu(\ell) &\geq  \frac{2L^2}{|\partial \Omega|} + \frac{|\partial \Omega|}{2}\left\{  \frac{2L}{|\partial \Omega|} \right\} -  \frac{|\partial \Omega|}{2}\left\{  \frac{2L}{|\partial \Omega|} \right\}^2 \\
&=\frac{2L^2}{|\partial \Omega|} + |K| \left(1 - \frac{2 |K|}{|\partial \Omega|}\right).
\end{align*}
There exists an alternative proof that argues as follows: if $\mathcal{L}$ is comprised of $k$ copies of $\partial \Omega$ and a line segment $K$, then ($\mu-$almost all) lines intersect $\mathcal{L}$ either exactly $2k$ or $2k+1$ times. The precise likelihood of each event is then uniquely determined by the expectation, which by Crofton is completely determined by the length, and, going through the proof of the Lemma, we see that we have equality.\\

\textit{3. An upper bound.} It remains to prove that
$$ E_{\Omega}(L) - \frac{2L^2}{|\partial \Omega|}  \leq \frac{|\partial \Omega|}{4}.$$
Let us assume that $L = (k+\alpha) |\partial \Omega|$, where
$k \in \mathbb{N}$ and $0 \leq \alpha < 1$. We take $\mathcal{L}$ to be $k$ copies of the boundary $\partial \Omega$,
then split the boundary into many small segments of equal length and add each of these small sets to $\mathcal{L}$ with likelihood given by $\alpha$. Taking a limit, we see that in expectation we
expect for a random construction of this type that
\begin{align*}
  \mathbb{E} \frac{1}{4} \int n_{\ell}(\mathcal{L})^2 d\mu(\ell) &=    \frac{1}{4} \int X \cdot 1_{\ell \cap \Omega \neq \emptyset} d\mu(\ell) = \frac{X}{2} |\partial \Omega|
  \end{align*}  
where
$$ X = \left[ (1-\alpha)^2 (2k)^2 + 2\alpha(1-\alpha) (2k+1)^2 + \alpha^2 (2k+2)^2 \right]$$
is the likelihood of a random line passing through $2k, 2k+1$ and $2k+2$ points on the boundary, respectively.
We obtain
\begin{align*}
  \mathbb{E} \frac{1}{4} \int n_{\ell}(\mathcal{L})^2 d\mu(\ell) &\leq (2k^2 + 4\alpha k + \alpha^2 + \alpha) |\partial \Omega|.
  \end{align*}  
  Recalling $L = (k+\alpha) |\partial \Omega|$, we obtain
  $$ E_{\Omega}(L) \leq \frac{2 L^2}{|\partial \Omega|} + (\alpha - \alpha^2) |\partial \Omega| \leq \frac{2 L^2}{|\partial \Omega|} +  \frac{ |\partial \Omega|}{4}.$$
  
\textit{4. The case of higher dimensions.} We use (see \cite{stein}) that for two universal constants $c_n, c_n^* > 0$ that only depend on the
dimension
$$ c_n \int n_{\ell}(\mathcal{L})^2 d\mu(\ell) - \mathcal{H}^{n-1}(\mathcal{L}) = c_n^* \int_{\mathcal{L}} \int_{\mathcal{L}}\frac{\left|\left\langle n(x), y - x \right\rangle  \left\langle  y - x, n(y) \right\rangle   \right| }{\|x - y\|^{n+1}}~d \sigma(x) d\sigma(y)$$
and thus, having fixed the surface area, minimizing either of these quantities is equivalent to minimizing the other. We try to minimize the quadratic Crofton functional. Crofton's formula implies that fixing $ \mathcal{H}^{n-1}(\mathcal{L}) $ is fixing the expected number of intersections. Appealing to the Lemma, we see that the expected squared intersection number is minimized if $ n_{\ell}(\mathcal{L})$ assumes at most 2 (adjacent) values. For convex $\Omega \subset \mathbb{R}^n$ taking $m$ copies of the boundary $\partial \Omega$ and possibly adding a hyperplane segment leads to a set for which the intersection number $ n_{\ell}(\mathcal{L})$ is either $2m$ or $2m+1$ for almost all lines $\ell$ and we conclude the result.
\end{proof}

\textbf{Acknowledgments.} I am grateful to Alan Chang for pointing out \cite{chang} and valuable discussions.

\end{document}